\documentclass{article}

\usepackage{amsmath,amsthm,amsfonts}

\newtheorem{theorem}{Theorem}

\theoremstyle{definition}

\newtheorem{remark}[theorem]{Remark}
\newtheorem{example}[theorem]{Example}

\newcommand{\vA}{\mathbf{A}}
\newcommand{\vf}{\mathbf{f}}
\newcommand{\vzero}{\mathbf{0}}
\newcommand{\vbeta}{\boldsymbol{\beta}}
\newcommand{\vgamma}{\boldsymbol{\gamma}}

\newcommand{\T}{\mathrm{T}}

\usepackage{geometry}

\usepackage{authblk}

\title{Differentiating a Linear Recursive Sequence}
\author[1]{D{\'a}vid Papp\thanks{Email: dpapp@ncsu.edu, ORCID: 0000-0003-4498-6417}}
\author[2]{Kolos Csaba {\'A}goston\thanks{Email: kolos.agoston@uni-corvinus.hu, ORCID: 0000-0002-6738-0592}}
\affil[1]{North Carolina State University, Raleigh, NC, USA}
\affil[2]{Corvinus University of Budapest, Budapest, Hungary}

\begin{document}



\maketitle

\begin{abstract}
Consider a sequence of real-valued functions of a real variable given by a homogeneous linear recursion with differentiable coefficients. We show that if the functions in the sequence are differentiable, then the sequence of derivatives also satisfies a homogeneous linear recursion whose order is at most double the order of original recursion. Similarly to the well-known operations that determine the elementwise sum and product of two linear recursive sequences, the coefficient functions of our recursion for the derivatives are easily computable from the original coefficient functions and their derivatives by direct manipulation of the coefficients of the characteristic polynomial of the recursion, without determining the roots. A simple application, computing linear recursions for derivatives of orthogonal polynomials, is presented.
\end{abstract}

\section{Introduction: the ``calculus'' of linear recurrence relations}
A \emph{homogeneous linear recurrence relation} of \emph{order} $k$ is a relation of the form
\begin{equation}\label{eq:def}
a_n = \alpha_1 a_{n-1} + \alpha_2 a_{n-2} + \dots + \alpha_k a_{n-k}\;\; (n\geq k);
\end{equation}
the \emph{coefficients} $\alpha_1,\dots, \alpha_k$ along with the \emph{initial values} $a_0,\dots,a_{k-1}$ determine the entire sequence $(a_i)_{i=1,2,\dots}$. Particularly important for this note is the case when the coefficients and the initial values are real-valued functions, rather than real numbers. For example, \emph{Chebyshev polynomials of the first kind} can be defined by the recursion
\begin{equation}\label{eq:ChebyshevT}
T_0(x) = 1, \quad T_1(x) = x, \quad T_n(x) = 2x T_{n-1}(x) - T_{n-2}(x)\;\;(n \geq 2).
\end{equation}

The \emph{characteristic polynomial} associated with the relation \eqref{eq:def} is the degree-$k$ univariate polynomial
\begin{equation*}
p(t) = t^k - \alpha_1 t^{k-1} - \alpha_2 t^{k-2} - \dots - \alpha_{k-1}t - \alpha_k.
\end{equation*}

Linear recurrence relations are ubiquitous, from the Fibonacci sequence to sequences of orthogonal polynomials. Hence, it is no surprise that there is a well-developed ``calculus'' to work with them, even in cases when the well-known formula for the general solution of \eqref{eq:def} involving the powers of the roots of the characteristic polynomial (see, e.g., \cite[E.q.~(1.4)]{EverestEtal2003}) is impractical. For example, it is quite well-known that given two linear recurrence relations
\begin{align*}
a_n &= \alpha_1 a_{n-1} + \dots + \alpha_k a_{n-k}\;\; (n\geq k) \text{ and }\\
b_n &= \beta_1 b_{n-1} + \dots + \beta_\ell b_{n-\ell}\;\; (n\geq\ell),
\end{align*}
we can easily compute the coefficients of a linear recurrence relation satisfied by the sums $(a_n+b_n)$, and same is true for the product sequence $(a_nb_n)$ \cite[Chap.~4]{EverestEtal2003}. The ``sum rule'' is especially simple: if $p_a$ and $p_b$ denote the characteristic polynomials of $(a_n)$ and $(b_n)$, respectively, then $(a_n+b_n)$ satisfies the recurrence whose characteristic polynomial is $p_a p_b$. When $p_a$ and $p_b$ share at least one root, the result can be further sharpened to the least common multiple instead of the product, $\operatorname{lcm}(p_a,p_b)$.

Note that the coefficients of the polynomials $p_a p_b$ and $\operatorname{lcm}(p_a,p_b)$ can be computed directly from the coefficients of $p_a$ and $p_b$; the proof of the sum rule does involve the roots of the characteristic polynomial, but the end result does not.

The ``product rule'' for linear recurrence relations is more complicated, but the end result is of the same flavor: $(a_nb_n)$ satisfies a linear recurrence relation of order no greater than $k\ell$, and the coefficients of such a relation can be computed as suitable polynomials of the original coefficients, expressed in terms of a resultant.
This, along with the state-of-the-art on elementwise (or Hadamard) operations on linear recurrence sequences in general, is presented in \cite[Chap.~4]{EverestEtal2003}).

In the remainder of this note, we show that a similar result holds for the sequence of derivatives when the recurrence \eqref{eq:def} defines a sequence of differentiable functions. Although both the result and our proof are elementary (the main challenge is to avoid any reference to the roots of the characteristic polynomial, which in this setting are not necessarily differentiable functions), we were surprised that we could not find any reference to them in the literature.

\section{A linear recurrence for the sequence of derivatives}

Throughout this section, we consider linear recurrence relations of the form \eqref{eq:def} in which the coefficients $\alpha_i$ are differentiable functions of a variable $x$. All derivatives of every function refer to derivatives with respect to $x$. For example, in the statement of \mbox{Theorem \ref{thm:main}} below, $p' = \partial p/\partial x$ is the polynomial we obtain from $p$ by differentiating each coefficient with respect to $x$, rather than the polynomial variable $t$. Additionally, $\gcd$ denotes the greatest common divisor of polynomials. We are now ready to state our main result.

\begin{theorem}\label{thm:main}
Suppose that the sequence of functions $f_0, f_1, \dots$ is defined by a linear recurrence relation of order $k$ whose characteristic polynomial is $p$, and assume that both the initial values $f_0,\dots,f_{k-1}$ and the coefficients in the recursion are differentiable. Then each $f_i$ is differentiable, and the sequence of derivatives $f'_0, f'_1, \dots$ satisfies the linear recurrence relation whose characteristic polynomial is $p^2 / \gcd(p,p')$.
\end{theorem}
\begin{proof}
The differentiability of $f_k, f_{k+1}, \dots$ is immediate. Let the homogeneous linear recurrence relation defining our sequence be
\begin{equation}\label{eq:fn-rec}
f_n = \sum_{\ell=1}^k \alpha_\ell f_{n-\ell} \quad (n\geq k).
\end{equation}
Introducing the notation $\alpha_0(\cdot) := -1$, we can now write
\begin{equation}\label{eq:f}
\sum_{\ell=0}^k \alpha_\ell f_{n-i-\ell} = 0 \qquad (i=0,\dots,k)
\end{equation}
for every $n\geq 2k$ by substituting $f_{n-i}$ for $f_n$ in the original recursion \eqref{eq:fn-rec}. Differentiating both sides of this equation, we also get the system of equations
\begin{equation}\label{eq:f'}
\sum_{\ell=0}^k \alpha'_\ell f_{n-i-\ell} + \alpha_\ell f'_{n-i-\ell} = 0 \qquad (i=0,\dots,k).
\end{equation}
Note that all subscripts are nonnegative as long as $n \geq 2k$.

Our goal is to find a linear combination of these equations from which every $f_{n-i}$ term cancels out, leaving only a homogeneous linear recursion on the derivatives, $f'_{n-i}$.

Let $\vA$ be the ${(k+1)\times(2k+1)}$ matrix
\begin{equation*}
\vA =
\begin{pmatrix}
\alpha_0 & \alpha_1 & \cdots & \alpha_k \\
& \alpha_0 & \cdots & \alpha_{k-1} & \alpha_k \\
&&\ddots & \ddots & \ddots & \ddots\\
&&&\alpha_0 & \alpha_1 & \cdots & \alpha_k
\end{pmatrix},
\end{equation*}
let $\vf$ be the $(2k+1)$-vector $(f_n,\dots,f_{n-2k})^\T$,
and let $\vA'$ and $\vf'$ denote the matrix and vector we obtain by differentiating each entry of $\vA$ and $\vf$, respectively. (To keep the notation light, we continue to omit the variable $x$ of the functions $f_i$ and $\alpha_i$, but we shall keep in mind throughout that $\vA$ and $\vf$ are a matrix and a vector of differentiable functions, and all derivatives refer to differentiation with respect to this variable.) With this notation the system of equations \eqref{eq:f}-\eqref{eq:f'} can be written as
\begin{equation*}
\begin{pmatrix}
\vA  & \vzero\\
\vA' & \vA 
\end{pmatrix}
\begin{pmatrix}
\vf\\
\vf' 
\end{pmatrix} = \vzero.
\end{equation*}
Consider now the following general linear combination of these equations:
\begin{itemize}
\item multiply Eq.~\eqref{eq:f} by $\beta_i$ for each $i\in\{0,\dots,k\}$,
\item multiply Eq.~\eqref{eq:f'} by $\gamma_i$ for each $i\in\{0,\dots,k\}$, and
\item add the resulting $2k+2$ equations.
\end{itemize}
This yields an identity with no $f_{n-i}$ terms precisely when
\begin{equation}\label{eq:nullspace}
\vbeta^\T\vA = \vgamma^\T\vA',
\end{equation}
wherein $\vbeta=(\beta_0,\dots,\beta_k)$ and $\vgamma=(\gamma_0,\dots,\gamma_k)$, and in this case the resulting identity is simply
\begin{equation}\label{eq:f'-matrix}
\vgamma^\T\vA\vf' = 0.
\end{equation}

There is yet another way in which we can write these equations. If we treat the variable vectors $\vbeta$ and $\vgamma$ as coefficient vectors for the polynomials
\[ \beta(t) = \beta_0 t^k + \beta_1 t^{k-1} + \dots + \beta_k \]
and
\[ \gamma(t) = \gamma_0 t^k + \gamma_1 t^{k-1} + \dots + \gamma_k \]
in the same fashion as how the characteristic polynomial $p$ is associated with the vector $(\alpha_0,\dots,\alpha_k)$,
we notice that \eqref{eq:nullspace} can be written as an equation between two polynomials:
\begin{equation}\label{eq:nullspace-poly}
\beta p = -\gamma p',
\end{equation}
since the $i$th equation of the system \eqref{eq:nullspace} (indexing them from $0$ to $2k$) is either
\[
\sum_{\ell=0}^i \beta_\ell \alpha_{i-\ell} = \sum_{\ell=0}^i \gamma_\ell \alpha'_{i-\ell}
\qquad \text{if } i\in\{0,\dots,k\}
\]
or
\[
\sum_{\ell=i-k}^k \beta_{\ell} \alpha_{i-\ell} = 
\sum_{\ell=i-k}^k \gamma_{\ell} \alpha'_{i-\ell} 
\qquad \text{if } i\in\{k, \dots, 2k\}.
\]
In either case, the two sides of the equation are the coefficients of $t^i$ in the polynomials $\beta p(t)$ on the left-hand side and $\gamma p'(t)$ on the right, respectively. Therefore \eqref{eq:nullspace} is equivalent to the polynomial equation $\beta p(t) = \gamma p'(t)$ for every $t$.

In the same vein, the characteristic polynomial of the recursion \eqref{eq:f'-matrix} on the derivatives is simply $\gamma p$.

The functional equation \eqref{eq:nullspace-poly} has a trivial non-zero solution: $(\beta,\gamma) = (-p',p)$, which leads to the characteristic polynomial $\gamma p = p^2$. That is, $p^2$ is the characteristic polynomial of a linear recurrence relation satisfied by the sequence of derivatives. This relation is of order $2k$, but we may be able to do slightly better. By the unique factorization theorem, \eqref{eq:nullspace-poly} will not have any lower-degree (nonzero) solutions unless $p$ and $p'$ have a common factor. However, if they do have a common factor $q$, then we have an additional solution, since we can also choose $(\beta,\gamma) = (-p'/q,p/q)$, which in turn yields the characteristic polynomial $\gamma p = p^2/q$ for the recursion of the derivatives.

To get the recurrence of the lowest order, we can choose $q$ to be $\gcd(p,p')$, which yields the characteristic polynomial $p^2/\gcd(p,p')$ for a recursion satisfied by the derivatives, as claimed.
\end{proof}

\begin{remark}\label{rem:initial}
The obtained recurrence relation on the derivatives depends only on the coefficients in the original recurrence, but not on the initial values. Furthermore, similarly to the sum and product rules, the new coefficient functions are easily computable from the original coefficient functions, without having to compute the roots of $p$. Recall that the greatest common divisor of two univariate polynomials can be computed using the Euclidean Algorithm \cite[Sec.1.5]{CLO2007}.
\end{remark}

\begin{example}[Chebyshev polynomials]
Let us return to the Chebsyhev polynomials of the first kind given in Eq.~\eqref{eq:ChebyshevT} and their counterparts, Chebyshev polynomials of the \emph{second kind}, defined by the recursion
\begin{equation}\label{eq:ChebyshevU}
U_0(x) = 1, \quad U_1(x) = 2x, \quad U_n(x) = 2x U_{n-1}(x) - U_{n-2}(x)\;\;(n \geq 2).
\end{equation}
Note that the recurrence relation is the same for both $(T_n)$ and $(U_n)$, only their initial values differ. It is well-known (see, e.g., \cite[Sec.~2.4.5]{MasonHandscomb2002}) that their derivatives are closely related, namely
\begin{equation}\label{eq:T'U'}
T'_n(x) = n U_{n-1}(x) \text{ and } U'_n(x) = \frac{(n+1)T_{n+1}(x)-xU_n(x)}{x^2-1}.
\end{equation}
Since $a_n=n$ satisfies a homogeneous linear recursion of order $2$, we could in principle determine a linear recursion for $(T'_n) = (n U_{n-1})$ using the product rule mentioned in Section 1, although the calculations are quite involved. To derive a recurrence relation for $(U'_n)$ from Eq.~\eqref{eq:T'U'}, we would need both the sum and the product rules of \mbox{Section 1}, and we would obtain a relation of considerably higher order.

Instead, we can obtain the simplest recurrence relations (which is the same relation for both sequences according to Remark~\ref{rem:initial}) directly from Theorem~\ref{thm:main}. The characteristic polynomial of the recurrence relations in \eqref{eq:ChebyshevT} and \eqref{eq:ChebyshevU} is
\[ p(t) = t^2 - 2x t + 1. \]
The derivative of the characteristic polynomial is $p'(t)=-2t$ (recall that the prime superscript represents differentiation with respect to $x$), which has no non-constant common divisor with $p$. The squared characteristic polynomial is
\[p^2(t) = t^4 - 4x t^3 + (2 + 4 x^2)t^2  - 4x t + 1,\] therefore both sequences of derivatives, $(T_n')$ and $(U_n')$ satisfy the linear recurrence relation
\[ f_n(x) = 4x f_{n-1}(x) - (2+4x^2) f_{n-2}(x) + 4x f_{n-3}(x) - f_{n-4}(x)\quad (n\geq 4).\]
\end{example}

\begin{example}[Second derivative sequences]
The sequence $(x^n)_{n=0,1,\dots}$ satisfies a homogeneous first-order recurrence with $\alpha_1(x) = x$; its characteristic polynomial is $p(t) = t-x$. Its derivative sequence $(n x^{n-1})_{n=0,1,\dots}$ therefore satisfies the second-order recursion whose characteristic polynomial is $p^2(t) = (t-x)^2$. Because this polynomial has a double root, the second derivative sequence $(n(n-1) x^{n-2})_{n=0,1,\dots}$ satisfies not only a fourth-order, but a third-order recursion, whose characteristic polynomial is $(t-x)^3$.

Generally, if the characteristic polynomial $p$ of a sequence is of degree $k$ and has twice differentiable coefficients, then the sequence of first derivatives satisfies the linear recursion of order $2k$ whose characteristic polynomial is $p^2$, and the order of this recursion may be the lowest possible, as $p$ and $p'$ may not have a non-constant common divisor. However, the second derivatives satisfy a linear recursion of order at most $3k$, since $p$ is always a common divisor of $p^2$ and $(p^2)' = 2pp'$.
\end{example}

\begin{example}
The numerator $\gcd(p,p')$ may be non-constant even if $p$ does not have a ``generic'' double root, but has a constant (with respect to $x$) root, as the example of the recurrence relation
\[f_n = (x+1) f_{n-1}-x f_{n-2} \quad (n\geq 2)\]
shows: the characteristic polynomial is
\[p(t) = t^2 - (x+1) t + x = (t-x)(t-1), \]
which (for a generic $x$) has single roots. Yet, $p'(t) = -(t-1)$, and $\gcd(p,p') = (t-1)$, meaning that the derivative sequence satisfies a linear recurrence relation of $3$ rather than $4$:
\[ f'_n = (2x+1)f'_{n-1} - x(x+2) f'_{n-2} + x^2 f'_{n-3} \quad (n\geq 3). \]
\end{example}

\section*{Acknowledgments}
Research was done while D{\'a}vid Papp was visiting the Corvinus Institute of Advanced Studies, Corvinus University, Budapest, Hungary.

The authors are indebted to J{\'a}nos T{\'o}th (Budapest University of Technology) for helpful discussions and his comments on an early version of this manuscript.

Funding (D{\'a}vid Papp): This material is based upon work supported by the National Science Foundation under Grant No. DMS-1847865. This material is based upon work supported by the Air Force Office of Scientific Research under award number FA9550-23-1-0370. Any opinions, findings and conclusions or recommendations expressed in this material are those of the author(s) and do not necessarily reflect the views of the U.S. Department of Defense.

Funding (Kolos Csaba \'Agoston): Kolos Csaba \'Agoston's research was supported by the National Research, Development and Innovation Office under Grant FK 145838.

\bibliographystyle{vancouver}
\bibliography{references}


\vfill\eject

\end{document}